\documentclass[10pt]{article}

\usepackage{amssymb,amsthm,amsmath,hyperref}

\usepackage{graphicx,float}

\numberwithin{equation}{section}

\newtheorem{thm}{Theorem}[section]
\newtheorem{lem}{Lemma}[section]

\newtheorem{prop}{Proposition}[section]
\theoremstyle{definition}
\newtheorem{defn}{Definition}[section]
\theoremstyle{remark}
\newtheorem{rem}{Remark}[section]

\allowdisplaybreaks

\title{Global existence in critical spaces for   incompressible viscoelastic fluids\thanks{This work is supported partially by NSFC No.10871175,
10931007, 10901137, Zhejiang Provincial Natural Science Foundation
of China Z6100217,  and SRFDP No. 20090101120005.}}

\author{Ting Zhang,\thanks{zhangting79@hotmail.com}
   Daoyuan Fang\thanks{dyf@zju.edu.cn} \\\textit{\small Department of Mathematics, Zhejiang University,
Hangzhou 310027, China}}
\date{}
\begin{document}

\maketitle

\begin{abstract}
We investigate local and global strong solutions for the
incompressible viscoelastic system of Oldroyd--B type. We obtain the
existence and uniqueness of a solution in a functional setting
invariant by the scaling of the associated equations. More
precisely, the initial velocity has the same critical regularity
index as for the incompressible Navier--Stokes equations, and one
more derivative is needed for the deformation tensor. We point out a
smoothing effect on the velocity and a $L^1-$decay on the difference
between the deformation tensor and the identity matrix. Our result
implies that the deformation tensor $F$ has the same regularity as
the density of the compressible Navier--Stokes equations.
\end{abstract}

\section{Introduction}
In this paper, we consider the following system describing
incompressible viscoelastic fluids.
\begin{equation}
  \left\{\begin{array}{l}
    \nabla \cdot v=0, \ \ x\in \mathbb{R}^N,\ N\geq 2,\\
        v_t+v\cdot \nabla v+\nabla p=\mu\Delta v+\nabla\cdot\left[
        \frac{\partial W(F)}{\partial F}F^\top
        \right],\\
            F_t+v\cdot \nabla F=\nabla vF,\\
             F(0,x)=I+E_0(x),
   \ v(0,x)=v_0(x).
  \end{array}
  \right.\label{vis-E1.1}
\end{equation}
Here, $v$, $p$, $\mu>0$, $F$ and $W(F)$ denote, respectively, the
velocity field of materials, pressure, viscosity, deformation tensor
and elastic energy functional.  The third equation is simply the
consequence of the chain law. It can also be regarded as the
consistence condition of the flow trajectories obtained from the
velocity field $v$ and also of those obtained from the deformation
tensor $F$  (\cite{Dafermos,Gurtin,Lei,Lin,Liu}). Moreover, on the
right-hand side of the momentum equation, $\frac{\partial
W(F)}{\partial F}$ is the Piola--Kirchhoff stress tensor and
$\frac{\partial W(F)}{\partial F}F^\top$ is the Cauchy--Green
tensor. The latter is the change variable (from Lagrangian to
Eulerian coordinates) form of the former one (\cite{Lei}). The above
system is equivalent to the usual Oldroyd--B model for viscoelastic
fluids in the case of infinite Weissenberg number (\cite{Larson}). On
the other hand, without the viscosity term, it represents exactly
the incompressible elasticity in Eulerian coordinates. We refer to
 \cite{Byron,Dafermos,Larson,Lin2,Lions,Liu} and their references for
the detailed derivation and physical background of the above system.

Throughout this paper, we will use the notations of
    $$
    (\nabla v)_{ij}=\frac{\partial v_i}{\partial x_j},
     \ (\nabla v F)_{ij}=(\nabla v)_{ik}F_{kj},
     \ (\nabla\cdot F)_{i}=\partial_j F_{ij},
    $$
and summation over repeated indices will always be understood.

For incompressible viscoelastic fluids, Lin et al.\cite{Lin} proved
the global existence of classical small solutions for the
two--dimensional case with the initial data $v_0,E_0=F_0-I\in
H^k(\mathbb{R}^2)$, $k\geq2$, by introducing an auxiliary vector
field to replace the transport variable $F$. Using the method in
 \cite{Kawashima} for the damped wave equation, they\cite{Lin} also
obtained the global existence of classical small solutions for the
three--dimensional case with the initial data $v_0,E_0\in
H^k(\mathbb{R}^3)$, $k\geq3$. Lei and Zhou\cite{Lei3} obtained the
same results for the two--dimensional case with the initial data
$v_0,E_0\in H^s(\mathbb{R}^2)$, $s\geq4$. via the incompressible
limit working directly on the deformation tensor $F$. Then, Lei et
al.\cite{Lei2} proved the global existence for two--dimensional
small-strain viscoelasticity with $H^2(\mathbb{R}^2)$ initial data
and without assumptions on the smallness of the rotational part of
the initial deformation tensor.  It is worth noticing that the
global existence and uniqueness for the large solution of the
two--dimensional problem is still open.  Recently, by introducing an
auxiliary function $w=\Delta v+\frac{1}{\mu}\nabla\cdot E$, Lei et
al.\cite{Lei} obtained a weak dissipation on the deformation $F$
and the global existence of classical small solutions to
$N$--dimensional system with the initial data $v_0,E_0\in
H^2(\mathbb{R}^N)$ and $N=2,3$. All these results need that the
initial data $v_0$ and $E_0$ have the same regularity, and the
regularity index of the initial velocity $v_0$ is bigger than the
critical regularity index for the classical incompressible
Navier--Stokes equations.  In this paper, using Danchin's method to
study the compressible Navier--Stokes system, we are concerned with
the existence and uniqueness of a solution for the initial data in a
functional space with minimal regularity order. Although
 the equation (\ref{vis-E1.1})$_3$ for the deformation tensor $F$ is identical to the equation for
the vorticity $\omega=\nabla\times v$ of the Euler equations
    $$
    \partial_t \omega+  v\cdot \nabla \omega=\nabla v\omega,
    $$
our result implies that the deformation tensor $F$ has the same
regularity as the density of the compressible Navier--Stokes
equations.

 At this stage, we will use scaling considerations for
 (\ref{vis-E1.1}) to guess which spaces may be critical. We observe
  that  (\ref{vis-E1.1}) is invariant by the transformation
    $$
    (v_0(x), F_0(x))\rightarrow (lv_0(lx),F_0(lx)),
    $$
        $$
(v(t,x), F(t,x),P(t,x))\rightarrow
(lv(l^2t,lx),F(l^2t,lx),l^2P(l^2t,lx)),
        $$
up to a change of the elastic energy functional $W$ into $l^2W$.

\begin{defn}\label{vis-D1.1}
  A functional space $E\subset (\mathcal{S}'(\mathbb{R}^N))^N\times (
  \mathcal{S}'(\mathbb{R}^N)^{N\times N}$ is called
a critical space if the associated norm is invariant under the
transformation $
    (v(x), F(x))\rightarrow (lv(lx),F(lx))
    $ (up to a constant
independent of $l$).
\end{defn}

Obviously
$(\dot{H}^{\frac{N}{2}-1})^N\times(\dot{H}^{\frac{N}{2}})^{N\times
N}$ is a critical space for the initial data. The space
$\dot{H}^{\frac{N}{2}}$ however is not included in $L^\infty$,   we
cannot expect to get $L^\infty$ control on the deformation tensor,
when we choose $(F_0-I)\in (\dot{H}^{\frac{N}{2}})^{N\times N}$.
Moreover, the product between functions does not extend continuously
from $\dot{H}^{\frac{N}{2}-1}\times\dot{H}^{\frac{N}{2}}$ to
$\dot{H}^{\frac{N}{2}-1}$, so that we will run into difficulties
when estimating the nonlinear terms. Similar to the compressible
Navier--Stokes system (\cite{Danchin}), we could use homogeneous Besov
spaces $B^s:=\dot{B}^s_{2 ,1}(\mathbb{R}^N)$ (refer to Sect.
\ref{vis-S3} for the definition of such spaces) with the same
derivative index. Now, $B^\frac{N}{2}$ is an algebra embedded in
$L^\infty$. This allows us to control the deformation tensor from
above without requiring more regularity on derivatives of $F$.
Moreover, the product is continuous from
$B^{\frac{N}{2}-\alpha}\times B^\frac{ N}{2}$ to
$B^{\frac{N}{2}-\alpha} $ for $0 \leq\alpha<N$, and
$(B^{\frac{N}{2}-1})^N\times(B^{\frac{N}{2}})^{N\times N}$ is a
critical space.

In this paper, we  assume that $E_0:=F_0-I$ and $v_0$ satisfy the
following constraints:
\begin{equation}
    \nabla \cdot v_0=0,
     \  \mathrm{det}(I+E_0)=1,
   \    \nabla\cdot E^\top_0=0,\label{vis-E1.3-0}
       \end{equation}
and
\begin{equation}
       \partial_m E_{0ij}-\partial_jE_{0im}
       =E_{0lj}\partial_lE_{0im}-E_{0lm}\partial_lE_{0ij}.
  \label{vis-E1.3}
\end{equation}
The first three of these expressions are just the consequences of
the incompressibility condition  and the last one can be understood
as the consistency condition for changing variables between the
Lagrangian and Eulerian coordinates \cite{Lei}.

For simplicity, we only consider the case of Hookean elastic
materials: $W(F)=\frac{1}{2}|F|^2=\frac{1}{2}\mathrm{tr}(FF^\top)$. Define the usual strain tensor by the form
    \begin{equation}
      E=F-I.
    \end{equation}
Then, the system (\ref{vis-E1.1}) is
\begin{equation}
  \left\{\begin{array}{l}
    \nabla \cdot v=0, \ \ x\in \mathbb{R}^N,\ N\geq 2,\\
        v_{it}+v\cdot \nabla v_i+\partial_i p=\mu\Delta v_i+ E_{jk}\partial_j E_{ik}+\partial_jE_{ij},\\
            E_t+v\cdot \nabla E=\nabla vE+\nabla v,\\
                (v,E)(0,x)=(v_0,E_0)(x).
  \end{array}
  \right.\label{vis-E1.4}
\end{equation}

Let us now state our main results. Define the following functional
space:
$$
    U^s_T=\left(L^1([0,T]; B^{s+1})\cap C([0,T];   B^{s-1})
    \right)^N\times \left( C([0,T]; B^s)
    \right)^{N\times N},$$
$$
    V^s= \left(L^1(\mathbb{R}^+; B^{s+1})\cap C(\mathbb{R}^+;   B^{s-1})
    \right)^N\times \left(L^2(\mathbb{R}^+; B^s)\cap C(\mathbb{R}^+; B^s\cap B^{s-1})
    \right)^{N\times N}.$$

\begin{thm}[Local result]\label{vis-T1.1}
    Suppose that the
   initial data satisfies the incompressible constraints (\ref{vis-E1.3-0}) and
 $E_0\in B^\frac{N}{2}$, $v_0\in
    B^{\frac{N}{2}-1}$.
    Then the following results hold true:
    \begin{description}
      \item[1)]  There exist $T>0$ and a unique local solution for system (\ref{vis-E1.4}) that satisfies
        $$
  (v,E)\in  \left(L^1([0,T]; B^{\frac{N}{2}+1})\cap C([0,T];   B^{\frac{N}{2}-1})
    \right)^N\times \left( C([0,T]; B^{\frac{N}{2}})
    \right)^{N\times N},
        $$
    \begin{equation}
      \|(v,E)\|_{U^\frac{N}{2}_T}\leq C(\|E_0\|_{B^\frac{N}{2}}+\|v_0\|_{B^{\frac{N}{2}-1}}),
    \end{equation}
  and
    \begin{equation}
       \nabla \cdot v =0,
     \  \mathrm{det}(I+E )=1,
   \    \nabla\cdot E^\top =0.
    \end{equation}
        \item[2)]  Moreover, if $E_0\in B^s$ and $u_0\in B^{s-1}$,
$s\in(\frac{N}{2},\frac{N}{2}+1)$, then
 \begin{equation}
      \|(v,E)\|_{U^s_T}
      \leq C(\|E_0\|_{B^s}+\|v_0\|_{B^{s-1}}).\label{vis-E1.8}
    \end{equation}
    \end{description}
\end{thm}

\begin{thm}[Global result]\label{vis-T1.2}
   Suppose that the
   initial data satisfies the incompressible constraints
   (\ref{vis-E1.3-0})--(\ref{vis-E1.3}),
    $E_0\in B^\frac{N}{2}\cap B^{\frac{N}{2}-1}$, $v_0\in   B^{\frac{N}{2}-1}$ and
        \begin{equation}
          \|E_0\|_{B^\frac{N}{2}\cap B^{\frac{N}{2}-1}}+\|v_0\|_{B^{\frac{N}{2}-1}}\leq \lambda,
        \end{equation}
        where $\lambda$ is a small positive constant.
Then the following results hold true:
    \begin{description}
      \item[1)]  There
  exists a unique global   solution for system (\ref{vis-E1.4}) that satisfies
    $$
v\in\left(L^1(\mathbb{R}^+; B^{\frac{N}{2}+1})\cap C(\mathbb{R}^+; B^{\frac{N}{2}-1})
    \right)^N ,
    $$
    $$
 E \in  \left(L^2(\mathbb{R}^+; B^\frac{N}{2})\cap C(\mathbb{R}^+; B^\frac{N}{2}\cap B^{\frac{N}{2}-1})
    \right)^{N\times N},
    $$
    \begin{equation}
      \|(v,E)\|_{V^\frac{N}{2}}\leq M(\|E_0\|_{B^\frac{N}{2}\cap B^{\frac{N}{2}-1}}+\|v_0\|_{B^{\frac{N}{2}-1}}).
    \end{equation}
    \item[2)] Moreover, if $E_0\in B^s$ and $u_0\in B^{s-1}$,
$s\in(\frac{N}{2},\frac{N}{2}+1)$, then
 \begin{equation}
      \|(v,E)\|_{V^s}
      \leq C(\|E_0\|_{B^s}+\|v_0\|_{B^{s-1}}).
    \end{equation}
    \end{description}
\end{thm}

\begin{rem}
  The $L^2$-decay in time for $E$ is a key point in the proof of
the global existence. We shall also get a $L^1-$decay in a space a
trifle larger than $B^\frac{ N}{2}$ (see Theorem \ref{vis-T6.1}
below).
\end{rem}

\begin{rem}\label{vis-r1.2}
  Theorem \ref{vis-T1.2} implies that the deformation tensor $F$ has similar
  property as the density of the compressible Navier--Stokes system
in   \cite{Danchin}. And we think that the
incompressible viscoelastic system is similar to the  compressible
Navier--Stokes system.
\end{rem}

\begin{rem}
Similar to the  compressible
  Navier--Stokes system in  \cite{Danchin},  the initial data do not
really belong to a critical space in the sense of Definition
\ref{vis-D1.1}. We indeed made the additional assumption $E_0\in
B^{\frac{ N}{2}- 1}$ (which however involves only low frequencies
and does not change the required local regularity for $E_0$). On the
other hand, our scaling considerations do not take care of the
  Cauchy--Green tensor term. A careful study of the linearized system (see
 Proposition
\ref{vis-P4.3} below) besides indicates that such an assumption may
be unavoidable.
\end{rem}
\begin{rem}
Considering the general viscoelastic model (\ref{vis-E1.1}), if the
strain energy function satisfies the strong Legendre--Hadamard
ellipticity condition
    \begin{equation}
    \frac{\partial^2W(I)}{\partial F_{il}\partial F_{jm}}=(\alpha^2-2\beta^2)\delta_{il}\delta{jm}+\beta^2(\delta_{im}\delta_{jl}
    +\delta_{ij}\delta_{lm}),
     \textrm{ with }\alpha>\beta>0,
    \end{equation}
 and the reference configuration stress--free condition
    \begin{equation}
    \frac{\partial W(I)}{\partial F}=0,
    \end{equation}
then we can obtain the same results as that in Theorems
\ref{vis-T1.1}--\ref{vis-T1.2}.
\end{rem}

In this paper, we introduce the following function:
    $$
    c=\Lambda^{-1}{ \nabla\cdot E},
    $$ where $\Lambda^{s} f=  \mathcal{F}^{-1} (|\xi|^s \hat{f})$.
Then, the system (\ref{vis-E1.4}) reads
    \begin{equation}
  \left\{\begin{array}{l}
    \nabla \cdot v=\nabla \cdot c=0, \ \ x\in \mathbb{R}^N,\ N\geq2,\\
        v_{it}+v\cdot \nabla v_i+\partial_i p=\mu\Delta v_i+ E_{jk}\partial_j E_{ik}+\Lambda c_i,\\
            c_t+v\cdot \nabla c +[\Lambda^{-1}\nabla\cdot, v\cdot ]\nabla E=\Lambda^{-1} \nabla\cdot(\nabla vE)-\Lambda v,\\
                     \Delta E_{ij}= \Lambda\partial_j c_i+\partial_k(\partial_k E_{ij}-\partial_jE_{ik }),\\
                (v,c)(0,x)=(v_0,\Lambda^{-1}\nabla\cdot E_0)(x).
  \end{array}
  \right.\label{vis-E1.14}
\end{equation}
So, we need to study the following mixed parabolic--hyperbolic
linear system with a convection term:
    \begin{equation}
  \left\{\begin{array}{l}
    \nabla \cdot v=\nabla \cdot c=0, \ \ x\in \mathbb{R}^N,\ N\geq2,\\
                               v_{it}+u\cdot \nabla v_i+\partial_i p-\mu\Delta v_i-\Lambda c_i=G,\\
     c_t+u\cdot \nabla c+\Lambda v =L,\\
                (v,c)(0,x)=(v_0,c_0)(x),
  \end{array}
  \right.
\end{equation}
where div$u=0$ and $c_0=\Lambda^{-1}\nabla\cdot E_0$. This system is
similar to the system
    \begin{equation}
      \left\{
      \begin{array}{l}
    c_t+v\cdot\nabla c+\Lambda d=F,\\
        d_t+v\cdot\nabla d-\bar{\mu}\Delta d-\Lambda c=G,
      \end{array}
      \right.
    \end{equation}
 in Danchin's paper ( \cite{Danchin}). Using the
method studying the compressible Navier--Stokes system  in
\cite{Danchin} (Proposition 2.3), we can obtain the Proposition
\ref{vis-P4.3} and $L^1-$decay on $E$.

When we finished this paper, we noticed that some
similar results were also obtained independently in \cite{Qian}, where J.Z. Qian  introduced the function
    $
    d_{ij}=-\Lambda^{-1}{ \nabla_j v_i},
    $ and considered the  mixed parabolic--hyperbolic
 system  of $(E, d)$.

As for the related studies on the existence of solutions to
nonlinear elastic systems, there are works by Sideris\cite{Sideris}
and Agemi\cite{Agemi} on the global existence of classical small
solutions to three--dimensional compressible elasticity, under the
assumption that the nonlinear terms satisfy the null conditions. The
global existence for three--dimensional incompressible elasticity
was then proved via the incompressible limit method (
\cite{Sideris2004}) and  by a different method ( \cite{Sideris2007}). It
is worth noticing that the  global existence and uniqueness for the
corresponding two--dimensional problem is still open.

Now, let us recall some classical results for the following
incompressible Navier--Stokes equations.
\begin{equation}
  \left\{\begin{array}{l}
        v_t+v\cdot \nabla v+\nabla p=\mu\Delta v,\\\nabla \cdot v=0, \\
   \ v(0,x)=v_0(x).
  \end{array}
  \right.\label{vis-E1.1-NS}
\end{equation}
In 1934, J. Leray proved the existence of global weak solutions for
(\ref{vis-E1.1-NS}) with divergence--free $u_0\in L^2$ (see
\cite{Leray}). Then, H. Fujita and T. Kato\cite{Fujita}
obtain the uniqueness with $u_0\in \dot{H}^{\frac{N}{2}-1}$. The
index $s= N/2-1$ is critical for (\ref{vis-E1.1-NS}) with initial
data in $\dot{H}^s$: this is the lowest index for which uniqueness
has been proved (in the framework of Sobolev spaces). This fact is
closely linked to the concept of scaling invariant space. Let us
precise what we mean. For all $l > 0$, system (\ref{vis-E1.1-NS}) is
obviously invariant by the transformation
$$u(t, x) \rightarrow u_l(t, x) := lu(l^2t, lx),
u_0(x)\rightarrow u_{0l}(x):= lu_0(lx),$$
 and a straightforward
computation shows that $\|u_0\|_{\dot{H}^{\frac{N}{2}-1}}
=\|u_{0l}\|_{\dot{H}^{\frac{N}{2}-1}}$. This idea of using a
functional setting invariant by the scaling of (\ref{vis-E1.1-NS})
is now classical and originated many works. Refer to
\cite{Cannone,Chemin,Chemin07-0,Koch01} for a recent panorama.

The rest of this paper is organized as follows. In Section
\ref{vis-S2}, we state  three lemmas describing the intrinsic
properties of viscoelastic system. In Section \ref{vis-S3}, we
present the functional tool box: Littlewood--Paley decomposition,
product laws in Sobolev and hybrid Besov spaces. The next section is
devoted to the study of some linear models associated to
(\ref{vis-E1.4}). In Section \ref{vis-s5}, we will study  the local
well-posedness for (\ref{vis-E1.4}).  At last, we will prove Theorem
\ref{vis-T1.2} in Section \ref{vis-s6}.

\section{Basic mechanics of viscoelasticity}\label{vis-S2}
Using the similar arguments as that in the proof of Lemmas 1--3 in
\cite{Lei}, we can easily obtain the following three lemmas.

\begin{lem}
  Assume that $\mathrm{det}(I+E_0)=1$ is satisfied and $(v,F)$ is the solution of system (\ref{vis-E1.4}). Then the following is always true:
  \begin{equation}
    \mathrm{det}(I+E)=1,
  \end{equation}
for all time $t\geq0$, where the usual strain tensor $E=F-I$.
\end{lem}

\begin{lem}\label{vis-L2.2}
  Assume that $\nabla\cdot E^\top_0=0$ is satisfied, then the solution $(v,F)$  of system (\ref{vis-E1.4}) satisfies the following identities:
  \begin{equation}
    \nabla \cdot F^\top=0, \ \textrm{ and }\ \nabla\cdot E^\top=0,
  \end{equation}
for all time $t\geq0$.
\end{lem}

\begin{lem}\label{vis-L2.3}
  Assume that  (\ref{vis-E1.3}) is satisfied and $(v,F)$ is the solution of system (\ref{vis-E1.4}). Then the following is always true:
  \begin{equation}
    \partial_m E_{ij}-\partial_jE_{im}
       =E_{lj}\partial_lE_{im}-E_{lm}\partial_lE_{ij},
  \end{equation}
for all time $t\geq0$.
\end{lem}

\section{Littlewood--Paley theory and Besov spaces}\label{vis-S3}
The proof of most of the results presented in this paper requires a
dyadic decomposition of Fourier variable (\textit{Littlewood--Paley
composition}). Let us briefly explain how it may be built in the
case $x\in \mathbb{R}^N$, $N\geq2$, (see
\cite{Danchin,Danchin2001,Danchin2003}).

 Let $\mathcal{S}(\mathbb{R}^N)$ be the Schwarz class. $\varphi(\xi)$ is a  smooth function valued in [0,1]
such that
    $$
    \textrm{supp}\varphi\subset\{\frac{3}{4}\leq|\xi|\leq\frac{8}{3}\}
    \ \textrm{ and }
     \sum_{q\in\mathbb{Z}}\varphi(2^{-q}\xi)=1,\  |\xi|\not=0.
    $$
Let $h(x)=(\mathcal{F}^{-1}\varphi)(x)$. For $f\in\mathcal{S'}$
(denote the set of temperate distributes, which is the dual one of
$\mathcal{S}$), we can define the homogeneous dyadic blocks as
follows:
        $$
        \Delta_{q}f(x):=\varphi(2^{-q}D)f(x)=2^{Nq}\int_{\mathbb{R}^N} h(2^{q}y)f(x-y)dy,
        \ \mbox{if}\        q\in\mathbb{Z},
        $$
where    $\mathcal{F}^{-1} $ represents the    inverse Fourier
transform. Define the low frequency cut-off by
    $$
    S_{q}f(x):=\sum_{p\leq
    q-1}\Delta_{p}f(x)=\chi(2^{-q}D)f(x).
    $$
 The  Littlewood--Paley
decomposition has nice properties of quasi-orthogonality,
    $$
    \Delta_{p}\Delta_{q}f_1\equiv 0,\ \mbox{if}\ |p-q|\geq 2,
    $$
and
        $$
        \Delta_{q}(S_{p-1}f_1\Delta_{p}f_2)\equiv 0,
               \ \mbox{if}\ |p-q|\geq 5.
        $$

The Besov space can be characterized in virtue of the
Littlewood--Paley decomposition.
\begin{defn}
Let $1\leq p\leq\infty$ and $s\in \mathbb{R}$. For $1\leq r\leq
\infty$, the Besov spaces $\dot{B}^{s}_{p,r}(\mathbb{R}^N)$,
$N\geq2$, are defined by
    $$
    f\in \dot{B}^{s}_{p,r}(\mathbb{R}^N)
    \Leftrightarrow \left\| 2^{qs}\|\Delta_{q}f\|_{L^{p}(\mathbb{R}^N)}
    \right\|_{l^r_q}<\infty
    $$
and $B^{s}(\mathbb{R}^N)=\dot{B}^{s}_{2,1}(\mathbb{R}^N)$.
\end{defn}

The definition of $\dot{B}^s_{p,r}(\mathbb{R}^N)$ does not depend on
the choice of the Littlewood--Paley decomposition.  Let us recall
some classical estimates in Sobolev spaces for the product of two
functions ( \cite{Danchin2005}).

\begin{prop}\label{vis-P3.1}
  Let $1\leq r,p,p_1,p_2\leq\infty$. Then following inequalities hold true:
    \begin{equation}
    \|uv\|_{\dot{B}^{s}_{p,r}}\lesssim \|u\|_{L^\infty}\|v\|_{\dot{B}^{s }_{p,r}}
    +\|v\|_{L^\infty}\|u\|_{\dot{B}^{s }_{p,r}},\ \textrm{ if  }\ s>0,
    \end{equation}
        \begin{equation}
    \|uv\|_{\dot{B}^{s_1+s_2-\frac{N}{p}}_{p,r}}\lesssim \|u\|_{\dot{B}^{s_1}_{p,r}}\|v\|_{\dot{B}^{s_2}_{p,\infty}},
    \ \textrm{ if }\ s_1,s_2<\frac{N}{p}\ \textrm{ and }\ s_1+s_2>0,
        \end{equation}
        \begin{equation}
    \|uv\|_{\dot{B}^{s}_{p,r}}\lesssim \|u\|_{\dot{B}^{s}_{p,r}}\|v\|_{\dot{B}^{\frac{N}{p}}_{p,\infty}\cap L^\infty},
    \ \textrm{ if }\ |s| <\frac{N}{p},
        \end{equation}
        \begin{equation}
    \|uv\|_{\dot{B}^{-\frac{N}{p} }_{p,\infty}}\lesssim \|u\|_{\dot{B}^{s}_{p,1}}
    \|v\|_{\dot{B}^{-s}_{p,\infty} },
    \ \textrm{ if }\  s  \in(-\frac{N}{p},\frac{N}{p}],\ p\geq2.
        \end{equation}
\end{prop}

We give the following definition of hybrid Besov norms as that in
 \cite{Danchin}.
\begin{defn}
 For $\mu>0$,  $r\in [1,+\infty]$ and $s\in
\mathbb{R}$, we denote
    $$
    \|u\|_{\widetilde{B}^{s,r}_\mu}:=\sum_{ q\in\mathbb{Z}}
     2^{qs} \max \{\mu,2^{-q}\}^{1-\frac{2}{r}}
     \|\Delta_qu\|_{L^2}.
     $$
\end{defn}
Let us remark that$\|f\|_{\widetilde{B}^{s,\infty}_\mu}\approx
\|f\|_{{B}^s\cap B^{s-1}}$ and $\|f\|_{\widetilde{B}^{s,2}_\mu}=
\|f\|_{{B}^s}$. Let us recall some  estimates in hybrid Besov spaces
for the product of two functions.

\begin{prop}[ \cite{Danchin}, Proposition 5.3]\label{vis-P3.2}
Let $r\in [1,\infty]$ and $s,t\in\mathbb{R}$. There exists some
constant C such that
    $$
    \|T_u v\|_{\tilde{B}^{s+t-\frac{N}{2},r}_\mu}\leq C\|u\|_{\tilde{B}^{s,r}_\mu}\|v\|_{B^t}, \textrm{ if }
    \ s\leq\min\{1-\frac{2}{r}+\frac{N}{2},\frac{N}{2}\},
    $$
  $$
    \|T_u v\|_{\tilde{B}^{s+t-\frac{N}{2},r}_\mu}\leq C\|u\|_{B^s}\|v\|_{\tilde{B}^{t,r}_\mu}, \textrm{ if }
    \ s\leq \frac{N}{2} ,
    $$
  $$
    \|R(u, v)\|_{\tilde{B}^{s+t-\frac{N}{2},r}_\mu}\leq C\|u\|_{\tilde{B}^{s,r}_\mu}\|v\|_{B^t}, \textrm{ if }
    \ s+t>\max\{0,1- \frac{2}{r} \},
    $$
  where  $$
        T_{f}g=\sum_{p\leq
        q-2}\Delta_{p}f\Delta_{q}g=\sum_{q}S_{q-1}f\Delta_{q}g
        $$
and    $$
    R(f,g)=\sum_{q}\Delta_{q}f\widetilde{\Delta}_{q}g\ \mbox{with}
     \ \widetilde{\Delta}_{q}:=\Delta_{q-1}+\Delta_{q}+\Delta_{q+1}.
     $$
\end{prop}

\begin{lem}\label{vis-L3.1}
  \begin{equation}
   \left\| [\Lambda^{-1}\nabla, u]\nabla v\right\|_{\tilde{B}^{\frac{N}{2},\infty}_\mu}
   \leq C\|\nabla
   u\|_{B^\frac{N}{2}}\|v\|_{\tilde{B}^{\frac{N}{2},\infty}_\mu}.
  \end{equation}
\end{lem}
\begin{proof}
  Since
        \begin{eqnarray*}
       \left|\mathcal{F}([\Lambda^{-1}\partial_i, u]\partial_j v)\right|&=&\left|\int_{\mathbb{R}^N}
       \frac{\xi_i}{|\xi|}\hat{u}(\xi-\eta)\eta_j\hat{v}(\eta)-\hat{u}(\xi-\eta)
       \frac{\eta_i}{|\eta|}\eta_j\hat{v}(\eta)
       d\eta\right|\\
            &=& \left|\int_{\mathbb{R}^N}
       \frac{\xi_i|\eta|-\eta_i|\xi|}{|\xi||\eta|}\eta_j\hat{u}(\xi-\eta)\hat{v}(\eta)
       d\eta\right|\\
                   &=&\left| \int_{\mathbb{R}^N}
       \frac{\xi_i(|\eta|-|\xi|)+(\xi_i-\eta_i)|\xi|}{|\xi||\eta|}\eta_j\hat{u}(\xi-\eta)\hat{v}(\eta)
       d\eta\right|\\
       &\leq&\int_{\mathbb{R}^N}
       |\eta-\xi||\hat{u}(\xi-\eta)||\hat{v}(\eta)|
       d\eta\\
        &=&\int_{\mathbb{R}^N}
       \widehat{(\nabla u)^*}(\xi-\eta)\widehat{v^*}(\eta)
       d\eta=\mathcal{F}((\nabla u)^*v^*),\\
        \end{eqnarray*}
where $f^*=\mathcal{F}^{-1}(|\hat{f}|)$, we have
    \begin{eqnarray*}
    \|\Delta_p  (  [\Lambda^{-1}\partial_i, u]\partial_j v))\|_{L^2}&=&C
\|\varphi(2^{-p}\xi) \mathcal{F} (  [\Lambda^{-1}\partial_i,
u]\partial_j v))\|_{L^2_\xi}\\
    &\leq& C \|\varphi(2^{-p}\xi)
\mathcal{F}((\nabla u)^*v^*)\|_{L^2_\xi}\\
&=&C \|\Delta_p  ( (\nabla u)^*v^*)\|_{L^2}
    \end{eqnarray*}
and from Proposition \ref{vis-P3.2},
    \begin{eqnarray*}
      \left\| [\Lambda^{-1}\nabla, u]\nabla v\right\|_{\tilde{B}^{\frac{N}{2},\infty}_\mu}
   &\leq& C\left\| (\nabla u)^*v^*\right\|_{\tilde{B}^{\frac{N}{2},\infty}_\mu}\nonumber\\
   &\leq& C\|(\nabla
   u)^*\|_{B^\frac{N}{2}}\|v^*\|_{\tilde{B}^{\frac{N}{2},\infty}_\mu}\nonumber\\
   &\leq& C\|\nabla
   u\|_{B^\frac{N}{2}}\|v\|_{\tilde{B}^{\frac{N}{2},\infty}_\mu},
    \end{eqnarray*}
where we use the face that $\|\Delta_p
f^*\|_{L^2}=\|\varphi(2^{-p}\xi)|\hat{f}|\|_{L^2_\xi}=
\|\varphi(2^{-p}\xi)\hat{f}(\xi)\|_{L^2_\xi}=\|\Delta_p f\|_{L^2}$.
\end{proof}

Then, we give the definition of the Chemin-Lerner type
spaces.

\begin{defn}
Let  $s\in\mathbb{R}$,
$(r,\lambda,p)\in[1,\,+\infty]^3$ and $T\in]0,\,+\infty]$. We define
$\widetilde{L}^{\lambda}_T(\dot B^s_{p\,r})$ as the completion
of $C([0,T],\mathcal{S})$ by the norm
$$
\| f\|_{\widetilde{L}^{\lambda}_T(\dot{B}^s_{p,r})} :=
\Big(\sum_{q\in\mathbb{Z}}2^{qrs} \Big(\int_0^T\|\Delta_q\,f(t)
\|_{L^p}^{\lambda}\,
dt\Big)^{\frac{r}{\lambda}}\Big)^{\frac{1}{r}} <\infty.
$$
with the usual change if $\lambda=\infty$, $r=\infty.$
\end{defn}

\section{A linear model with convection}\label{vis-S4}
Using the similar argument as that in the proof of Proposition A.1
in  \cite{Danchin2001}, we can obtain  the  following proposition.
\begin{prop} \label{vis-P4.1}
  Let $(p,r)\in[1,+\infty]^2$ and $s\in(-\frac{N}{p}, 1+\frac{N}{p})$, and $v$ be a solenoidal vector field such that $\nabla v\in L^1(0,T;\dot{B}^\frac{N}{p}_{p,r}\cap L^\infty)$. Suppose that
  $f_0\in \dot{B}^s_{p,r}$, $g\in L^1(0,T;\dot{B}^s_{p,r})$ and that $f\in L^\infty(0,T;\dot{B}^s_{p,r})\cap C([0,T];\mathcal{S}')$ solves
    $$\left\{
    \begin{array}{l}
      \partial_t f+\nabla\cdot(vf)=g,\\
      f|_{t=0}=f_0.
    \end{array}\right.
    $$
  Then there exists a constant $C$ depending only on $s$, $p$, $r$ and $N$,   such that the following inequality holds true for $t\in[0,T]$:
    \begin{equation}
      \|f\|_{\widetilde{L}^\infty_t(\dot{B}^s_{p,r})}\leq e^{C\widetilde{V}(t)}\left(
       \|f_0\|_{ \dot{B}^s_{p,r} }+\int^t_0e^{-C\widetilde{V}(t)}\|g(t)\|_{\dot{B}^s_{p,r}}dt
      \right),
    \end{equation}
  with $\widetilde{V}(t)=\int^t_0\|\nabla v\|_{\dot{B}^\frac{N}{p}_{p,r}\cap L^\infty}d\tau$. Moreover,
  if $r<\infty$, then $f\in C([0,T];\dot{B}^\frac{N}{p}_{p,r})$.
\end{prop}

\begin{prop}[ \cite{Danchin2003}, Prop. 3.2] \label{vis-P4.2}
  Let $s\in (-\frac{N}{2},2+\frac{N}{2})$, $r\in[1,\infty]$, $u_0$ be a divergence-free vector field with coefficients in $\dot{B}^{s-1}_{2,r}$ and $f\in \widetilde{L}^1_T(\dot{B}^{s-1}_{2,r})$.
  Let $u,v$ be two divergence-free time dependent vector fields  such that $\nabla v\in L^1(0,T;\dot{B}^{\frac{N}{2}}_{2,r}\cap L^\infty)$
   and $u\in C([0,T);\dot{B}^{s-1}_{2,r})\cap
  \widetilde{L}^1_T(\dot{B}^{s+1}_{2,r})$. Assume in addition that
    \begin{equation}
      \left\{\begin{array}{l}
        u_t+v\cdot\nabla u-\mu\Delta u+\nabla \Pi=f,\\
        \nabla \cdot u=0,\\
        u|_{t=0}=u_0,
      \end{array}
      \right.
    \end{equation}
  is fulfilled for some distribution $\Pi$. Then, there exists $C=C(s,r,N)$ such that the following estimate holds true for $t\in[0,T]$:
  \begin{eqnarray}
    &&\|u\|_{\widetilde{L}^\infty_t (\dot{B}^{s-1}_{2,r})}+\mu\|u\|_{\widetilde{L}^1_t(\dot{B}^{s+1}_{2,r})}
    +\|\nabla\Pi\|_{\widetilde{L}^1_t(\dot{B}^{s-1}_{2,r})}\nonumber\\
        &    \leq& e^{C\widetilde{V}(t)}\left(
    \|u_0\|_{\dot{B}^{s-1}_{2,r}}+C\|f\|_{\widetilde{L}^1_T(\dot{B}^{s-1}_{2,r})}
    \right),
  \end{eqnarray}
  where $\widetilde{V}(t)=\int^t_0\|\nabla v\|_{\dot{B}^\frac{N}{2}_{2,r}\cap L^\infty}d\tau$.
\end{prop}

After the change of function
    $$
    c=\Lambda^{-1}{ \nabla\cdot E},
    $$
the system (\ref{vis-E1.4}) reads (\ref{vis-E1.14}). At first, we
will study the following mixed linear system
    \begin{equation}
  \left\{\begin{array}{l}
    \nabla \cdot v=\nabla \cdot c=0, \ \ x\in \mathbb{R}^N,\ N\geq2,\\
                    v_{it}+u\cdot \nabla v_i+\partial_i p-\mu\Delta v_i-\Lambda c_i=G,\\
                     c_t+u\cdot \nabla c+\Lambda v =L,\\
                (v,c)(0,x)=(v_0,c_0)(x),
  \end{array}
  \right. \label{vis-E4.1}
\end{equation}
where div$u=0$ and $c_0=\Lambda^{-1}\nabla\cdot E_0$. Using the
similar arguments as that in  \cite{Danchin} (Proposition 2.3), we
can obtain the following proposition and omit the details. The main
different is that there is a pressure term $\nabla p$ in
(\ref{vis-E4.1})$_2$. Using that fact that $\nabla\cdot
v=\nabla\cdot c=0$, we have $\int v\cdot\nabla p = \int c\cdot\nabla
p =0$, and obtain the following proposition.
\begin{prop}\label{vis-P4.3}
  Let $(v,c)$ be a solution of (\ref{vis-E4.1}) on $[0,T)$, $1-\frac{N}{2}<\rho\leq 1+\frac{N}{2}$ and
  $\widetilde{U}(t)=\int^t_0\|u(\tau)\|_{B^{\frac{N}{2}+1}}d\tau$. The following estimate holds on $[0,T)$:
    \begin{eqnarray*}
     \|(v,c) \|_{X^\rho_T}\leq Ce^{C\widetilde{U}(T)}\big(&&
      \|c_0\|_{\tilde{B}^{\rho,\infty}_\mu}+\|v_0\|_{B^{\rho-1}}\\
      &&+\int^t_0e^{-C\widetilde{U}(s)}\left(
      \|L(s)\|_{
      \tilde{B}^{\rho,\infty}_\mu}+\|G(s)\|_{
    {B}^{\rho-1} }
      \right)ds
      \big),
    \end{eqnarray*}
where $C$ depends only on $N$ and $\rho$,
    \begin{eqnarray*}
    X^s_T=\big\{
    (v,c)\in && \left(L^1(0,T; {B}^{s+1} )\cap C(0,T; {B}^{s-1} )
    \right)^N \\
    && \times\left(L^1(0,T;\tilde{B}^{s,1}_\mu)\cap C(0,T;\tilde{B}^{s,\infty}_\mu)
    \right)^N
    \big\},
    \end{eqnarray*}
    and $\|(v,c) \|_{X^s_T}=
    \|v\|_{L^\infty_T( {B}^{s-1} )}+\mu\|v\|_{L^1_T( {B}^{s+1} )}
    +\|c\|_{L^\infty_T(\tilde{B}^{s,\infty}_\mu)}+\mu\|c\|_{L^1_T(\tilde{B}^{s,1}_\mu)}$.
\end{prop}

\section{Local results}\label{vis-s5}
\subsection{Local existence}
Let $\bar{v}$ be the solution of the linear heat equations,
    \begin{equation}
      \left\{\begin{array}{l}
        \partial_t\bar{v}-\mu\Delta\bar{v}=0,\\
        \bar{v}(0,x)=v_0(x).
      \end{array}
      \right.
    \end{equation}
It is easy to obtain that $\bar{v}\in C([0,T];B^{\frac{N}{2}-1} )$
and
    \begin{equation}
      \|\bar{v}\|_{L^r_T(B^{\frac{N}{2}-1+\frac{2}{r}} )}
      \leq C \|v_0\|_{B^{\frac{N}{2}-1} }
      \ r\in[1,\infty] .
    \end{equation}
    Then, we can choose $T_1\in (0,1)$ such that
    \begin{equation}
      \|\bar{v}\|_{L^1_{T_1}(B^{\frac{N}{2}+1} )}
      \leq \lambda_1,
    \end{equation}
where $\lambda_1$ is chosen in (\ref{vis-E5.7}) and
(\ref{vis-E5.9}).

Let $u=v-\bar{v}$. We will obtain the local existence of the
solution to the following system
\begin{equation}
  \left\{\begin{array}{l}
    \nabla \cdot u=0, \ \ x\in \mathbb{R}^N,\ N\geq 2,\\
        u_{it}+u\cdot \nabla u_i+u\cdot \nabla \bar{v}_i
        +\bar{v}\cdot \nabla u_i+\bar{v}\cdot \nabla \bar{v}_i
        +\partial_i p=\mu\Delta u_i+ E_{jk}\partial_j E_{ik}+\partial_jE_{ij},\\
            E_t+u\cdot \nabla E+\bar{v}\cdot \nabla E=\nabla uE+\nabla \bar{v}E+\nabla u+\nabla \bar{v},\\
                (u,E)(0,x)=( 0,E_0)(x).
  \end{array}
  \right.\label{vis-E4.2}
\end{equation}

Now, we use an iterative method to build approximate solutions
$(u^n,E^n)$ of (\ref{vis-E4.2}) which are solutions of linear
system. Set the first term $(E^0,u^0)$ to $(0,0)$. Then define
$\{(u^n,E^n)\}_{n\in\mathbb{N}}$ by induction.  Choose
$(E^{n+1},v^{n+1})$ as the solution of the following linear system
\begin{equation}
  \left\{\begin{array}{l}
    \nabla \cdot u^{n+1}=0, \ \ x\in \mathbb{R}^N,\ N\geq2,\\
        (u^{n+1})_{it}+(u^{n }+\bar{v})\cdot \nabla (u^{n+1})_i
         +\partial_i p^{n+1}-\mu\Delta (u^{n+1})_i\\
       \ \ \ \ \ \ \ \ \ \   \ \ \ \ \ =-(u^n+\bar{v})\cdot \nabla \bar{v}_i+ E^{n+1}_{jk}\partial_j E^{n+1}_{ik}+\partial_jE^{n+1}_{ij},\\
            E^{n+1}_t+(u^n+\bar{v})\cdot \nabla E^{n+1}=\nabla u^nE^{n}+\nabla \bar{v}E^{n}+\nabla u^n+\nabla \bar{v},\\
                (u^{n+1},E)(0,x)=( 0,E_0)(x).
  \end{array}
  \right.\label{vis-E4.3}
\end{equation}
The existence of the third equation can be obtained by the classical
result of the transport equations. About the second equation's
existence result, we can use the Friedrichs mollifiers. Here, we
omit the details.

    We are going to prove that, if $T\in(0,1)$ is small enough, the following bound holds for all $n\in\mathbb{N}$,
        \begin{equation*}
     \|E^n\|_{\widetilde{L}^\infty_T({B}^{\frac{N}{2}}) }
     \leq 6\|E_0\|_{ {B}^{\frac{N}{2}}},
       \   \|u^n\|_{\widetilde{L}^\infty_T(B^{s-1})}+ \|u^n\|_{L^1_T(B^{s+1})}\leq  \lambda_1.
    \eqno(P_n)
        \end{equation*}
        Suppose that ($P_n$) is satisfied and let us prove that $(P_{n+1})$ is also true.

From Propositions \ref{vis-P3.1} and \ref{vis-P4.1}, we get
    \begin{eqnarray}
     && \|E^{n+1}\|_{\widetilde{L}^\infty_T({B}^{\frac{N}{2}})}\nonumber\\
     &\leq& e^{C\widetilde{V}^n(T)}\left(
       \|E_0\|_{  {B}^{\frac{N}{2} } }+\int^T_0e^{-C\widetilde{V}^n(t)}\|\nabla u^nE^{n}+\nabla \bar{v}E^{n}+\nabla u^n+\nabla \bar{v}\|_{{B}^{\frac{N}{2}}}dt
      \right)\nonumber\\
        &\leq& e^{C\lambda_1}\left(
       \|E_0\|_{{B}^{\frac{N}{2}}}+
       \|(\nabla u^n,\nabla \bar{v})\|_{L^1_T(B^\frac{N}{2})}
       \|E^{n}\|_{{L}^\infty_T({B}^{\frac{N}{2}})}
       + \|(\nabla u^n,\nabla \bar{v})\|_{L^1_T(B^\frac{N}{2})} \right)\nonumber\\
      &\leq&   e^{C\lambda_1}
       \|E_0\|_{ {B}^{\frac{N}{2}} }+2\lambda_1 e^{C\lambda_1} \|E^{n}\|_{\widetilde{L}^\infty_T(\widetilde{B}^{\frac{N}{2},\infty}_\mu)}
      + 2\lambda_1e^{C\lambda_1},
    \end{eqnarray}
where $\widetilde{V}^n(t)=\int^t_0 \|\nabla u^n\|_{B^\frac{N}{2}}
+\|\nabla \bar{v}\|_{B^\frac{N}{2}}d\tau$. When $\lambda_1$
satisfies
    \begin{equation}
      2\lambda_1 e^{C\lambda_1} \leq \frac{1}{2},\  e^{C\lambda_1}\leq 2,
      \    2\lambda_1e^{C\lambda_1}\leq
      \|E_0\|_{ {B}^{\frac{N}{2}}},\label{vis-E5.7}
    \end{equation}
    we have
        $$
        \|E^{n+1}\|_{\widetilde{L}^\infty_T({B}^{\frac{N}{2}})}
        \leq 6\|E_0\|_{{B}^{\frac{N}{2}}}.
        $$

From Propositions \ref{vis-P3.1} and \ref{vis-P4.2}, we have
 \begin{eqnarray}
    &&\|u^{n+1}\|_{\widetilde{L}^\infty_T(B^{\frac{N}{2}-1})}+\mu\|u^{n+1}\|_{L^1_T(B^{\frac{N}{2}+1})}
    +\|\nabla p^{n+1}\|_{L^1_T(B^{\frac{N}{2}+1})}
     \nonumber\\
    &\leq& Ce^{C\widetilde{V}^n(t)} \|-(u^n+\bar{v})\cdot \nabla \bar{v}_i+ E^{n+1}_{jk}\partial_j E^{n+1}_{ik}+\partial_jE^{n+1}_{ij}\|_{L^1_T(B^{\frac{N}{2}-1})}\nonumber\\
        &\leq&Ce^{C\lambda_1} \left(\| (u^n+\bar{v})\|_{L^\infty_T
        (B^{\frac{N}{2}-1})} \|\nabla \bar{v}\|_{L^1_T(B^\frac{N}{2} )}\right.\nonumber\\
        &&\left.+
          T\|E^{n+1}\|_{L^\infty_T(B^{\frac{N}{2} })}^2+ T\| E^{n+1} \|_{L^\infty_T(B^{\frac{N}{2}})}\right)\nonumber\\
          &\leq&Ce^{C\lambda_1} \left(2\lambda^2_1+
          T(6\|E_0\|_{  {B}^{\frac{N}{2} }  })^2+
          6T\|E_0\|_{  {B}^{\frac{N}{2} }  }\right)
          \nonumber\\
            &\leq& \lambda_1,
  \end{eqnarray}
  where we choose $\lambda_1$ and $T$ satisfying
    \begin{equation}
      Ce^{C\lambda_1} \left(2\lambda^2_1+
          T(6\|E_0\|_{ {B}^{\frac{N}{2} }  })^2+
           6T\|E_0\|_{  {B}^{\frac{N}{2} }  }\right)
          \leq  \lambda_1.\label{vis-E5.9}
    \end{equation}
Thus, ($P_n$)    hold for all $n\geq0$. From (\ref{vis-E4.3}), we
can easily obtain that $ u^n_t,E^n_t $ are uniformly bounded in
$L^1_T(B^{\frac{N}{2}-1})$. Then, letting the limit of $\{u^n,
E^n\}$ is $(u,E)$, using the classical compactness arguments, we can
obtain that $(u+\bar{v},E)$ is the solution of (\ref{vis-E4.2}).
Thus, we can prove the existence part of Theorem \ref{vis-T1.1}.

\subsection{Further regularity property}
Let $s\in (\frac{N}{2},\frac{N}{2}+1)$. Under the additional
assumption $E_0\in B^s$ and $v_0\in B^{s-1}$, we obtain that
$\bar{v}\in C([0,T];B^{s-1} )$ and
    \begin{equation}
      \|\bar{v}\|_{L^r_T(B^{s-1+\frac{2}{r}} )}
      \leq C \|v_0\|_{B^{s-1} }
      \ r\in[1,\infty] .
    \end{equation}
   Then, we shall prove that the sequence $\{
(u^n,E^n)\}_{n\in \mathbb{N}}$ is uniformly bounded in $U^s_T$.

Applying Propositions \ref{vis-P3.1} and \ref{vis-P4.1}, and
$(P_n)$, we get
    \begin{eqnarray}
     && \|E^{n+1}\|_{\widetilde{L}^\infty_T({B}^{s})}\nonumber\\
     &\leq& e^{C\widetilde{V}^n(T)}\left(
       \|E_0\|_{  {B}^{s} }+\int^T_0e^{-C\widetilde{V}^n(t)}\|\nabla u^nE^{n}+\nabla \bar{v}E^{n}+\nabla u^n+\nabla \bar{v}\|_{{B}^{s}}dt
      \right)\nonumber\\
        &\leq& e^{C\lambda_1}\left(
       \|E_0\|_{{B}^{s}}+
       \|(\nabla u^n,\nabla \bar{v})\|_{L^1_T(B^\frac{N}{2})}
       \|E^{n}\|_{{L}^\infty_T({B}^{s})}\right.\nonumber\\
       &&\left.+
       \|(\nabla u^n,\nabla \bar{v})\|_{L^1_T(B^s)}
       \|E^{n}\|_{{L}^\infty_T({B}^{\frac{N}{2}})}
       + \|(\nabla u^n,\nabla \bar{v})\|_{L^1_T(B^s)} \right)\nonumber\\
      &\leq&   e^{C\lambda_1}
       \|E_0\|_{ {B}^{s} }+2\lambda_1 e^{C\lambda_1}
       \|E^{n}\|_{\widetilde{L}^\infty_T({B}^s)}+Ce^{C\lambda_1}
       \|v_0\|_{B^{s-1} }(6\|E_0\|_{{B}^\frac{N}{2}} +1)\nonumber\\
        &&+e^{C\lambda_1}
       \|u^n\|_{L^1_T(B^{s+1}) }(6\|E_0\|_{{B}^\frac{N}{2}} +1).\label{vis-E5.11}
    \end{eqnarray}
From Propositions \ref{vis-P3.1} and \ref{vis-P4.2}, we have
 \begin{eqnarray}
    &&\|u^{n+1}\|_{\widetilde{L}^\infty_T(B^{s-1})}
    +\mu\|u^{n+1}\|_{L^1_T(B^{s+1})}
     \nonumber\\
    &\leq& Ce^{C\widetilde{V}^n(t)} \|-(u^n+\bar{v})\cdot \nabla \bar{v}_i+ E^{n+1}_{jk}\partial_j E^{n+1}_{ik}+\partial_jE^{n+1}_{ij}\|_{L^1_T(B^{s-1})}\nonumber\\
        &\leq&Ce^{C\lambda_1} \left(\| (u^n+\bar{v})\|_{L^2_T(B^{\frac{N}{2}})} \|\nabla \bar{v}\|_{L^2_T(B^{s-1} )}
        \right.\nonumber\\
            &&\left.+
          T\|E^{n+1}\|_{L^\infty_T(B^{\frac{N}{2} })}
          \|E^{n+1}\|_{L^\infty_T(B^{s})}+ T\| E^{n+1} \|_{L^\infty_T(B^{s})}\right)\nonumber\\
          &\leq&Ce^{C\lambda_1} \lambda_1\|v_0\|_{B^{s-1}}+
          TCe^{C\lambda_1}\|E^{n+1}\|_{L^\infty_T(B^{s})}
          (6\|E_0\|_{  {B}^{\frac{N}{2} }  }+1).\label{vis-E5.12}
  \end{eqnarray}
From (\ref{vis-E5.11})--(\ref{vis-E5.12}), we have
     \begin{eqnarray}
    &&\|u^{n+1}\|_{\widetilde{L}^\infty_T(B^{s-1})}
    +\mu\|u^{n+1}\|_{L^1_T(B^{s+1})}+
    \frac{\mu\|E^{n+1}\|_{L^\infty_T(B^{s})}}{2e^{C\lambda_1}(6\|E_0\|_{{B}^\frac{N}{2}} +1)}
     \nonumber\\
          &\leq&Ce^{C\lambda_1} \lambda_1\|v_0\|_{B^{s-1}}+
          TCe^{C\lambda_1}\|E^{n+1}\|_{L^\infty_T(B^{s})}
          (6\|E_0\|_{  {B}^{\frac{N}{2} }  }+1)
         \nonumber\\
       && +\frac{ \mu\|E_0\|_{ {B}^{s} }}{2 (6\|E_0\|_{{B}^\frac{N}{2}} +1)}+\frac{C\mu}{2}
       \|v_0\|_{B^{s-1} }       +\frac{\lambda_1\mu }{(6\|E_0\|_{{B}^\frac{N}{2}} +1)}
       \|E^{n}\|_{\widetilde{L}^\infty_T({B}^s)}
        \nonumber\\
            &&+\frac{\mu}{2}
       \|u^n\|_{L^1_T(B^{s+1}) }.\label{vis-E5.13-1}
  \end{eqnarray}
From (\ref{vis-E5.7}), we have
    $ \frac{2\lambda_1 }{(6\|E_0\|_{{B}^\frac{N}{2}} +1)}\leq \frac{1
    }{2e^{C\lambda_1}(6\|E_0\|_{{B}^\frac{N}{2}} +1)}$. When $T_1\in (0,
    T]$ satisfies
    \begin{equation}
    2 T_1Ce^{C\lambda_1}
          (6\|E_0\|_{  {B}^{\frac{N}{2} }  }+1)\leq
             \frac{\mu}{4e^{C\lambda_1}(6\|E_0\|_{{B}^\frac{N}{2}} +1)},
    \end{equation}
we can prove by induction from the inequality (\ref{vis-E5.13-1})
that
     \begin{eqnarray}
    &&\|u^{n}\|_{\widetilde{L}^\infty_{T_1}(B^{s-1})}
    +\frac{\mu}{2}\|u^{n}\|_{L^1_{T_1}(B^{s+1})}+
    \frac{\mu\|E^{n}\|_{L^\infty_{T_1}(B^{s})}}{8e^{C\lambda_1}(6\|E_0\|_{{B}^\frac{N}{2}} +1)}
     \nonumber\\
          &\leq&Ce^{C\lambda_1} \lambda_1\|v_0\|_{B^{s-1}}
          +\frac{ \mu\|E_0\|_{ {B}^{s} }}{2 (6\|E_0\|_{{B}^\frac{N}{2}} +1)}
       +\frac{C\mu}{2}
       \|v_0\|_{B^{s-1} },\ \forall\ n.
  \end{eqnarray}
Now, we conclude that the sequence $\{ (u^n,E^n)\}_{n\in
\mathbb{N}}$ is uniformly bounded in $U^s_{T_1}$. This clearly
enables us to prove that the solution $(v,E)$ built in the previous
subsection also belongs to $U^s_{T_1}$.

From $(v,E)\in U^\frac{N}{2}_T$, we will prove that $(v,E)\in
U^s_T$. Applying Propositions \ref{vis-P3.1} and \ref{vis-P4.1}, we
get
    \begin{eqnarray}
     && \|E\|_{\widetilde{L}^\infty_T({B}^{s})}\nonumber\\
     &\leq& e^{C\widetilde{V}(T)}\left(
       \|E_0\|_{  {B}^{s} }+\int^T_0e^{-C\widetilde{V}(t)}\|\nabla vE+\nabla v\|_{{B}^{s}}dt
      \right)\nonumber\\
        &\leq& C\big(
       \|E_0\|_{{B}^{s}}+
      \int^T_0 \|\nabla v\|_{B^\frac{N}{2}}
       \|E\|_{{B}^{s}}dt+
       \|\nabla v\|_{L^1_T(B^s)}
       \|E\|_{{L}^\infty_T({B}^{\frac{N}{2}})}\nonumber\\
         &&
       + \|\nabla {v}\|_{L^1_T(B^s)} \big)\nonumber\\
         &\leq& C
       \|E_0\|_{{B}^{s}}+C
      \int^T_0 \|\nabla v\|_{B^\frac{N}{2}}
       \|E\|_{{B}^{s}}dt+C
       \|  v\|_{L^1_T(B^{s+1})}.\label{vis-E5.11-0}
    \end{eqnarray}
From Propositions \ref{vis-P3.1} and \ref{vis-P4.2}, we have
 \begin{eqnarray}
    &&\|v\|_{\widetilde{L}^\infty_T(B^{s-1})}
    +\mu\|v\|_{L^1_T(B^{s+1})}
     \nonumber\\
    &\leq& Ce^{C\widetilde{V}(t)}\left(\|v_0\|_{B^{s-1}}+ \|E_{jk}\partial_j E_{ik}+\partial_jE_{ij}\|_{L^1_T(B^{s-1})}
    \right)\nonumber\\
        &\leq&C\|v_0\|_{B^{s-1}}+C\int^T_0 ( \|E\|_{B^{\frac{N}{2} }}
          \|E\|_{B^{s}}+\| E
          \|_{B^{s}})dt\nonumber\\
          &\leq&C\|v_0\|_{B^{s-1}}+C\int^T_0 \| E
          \|_{B^{s}}dt.\label{vis-E5.12-0}
  \end{eqnarray}
From (\ref{vis-E5.11-0})--(\ref{vis-E5.12-0}), we have
     \begin{eqnarray}
    &&\|v\|_{\widetilde{L}^\infty_T(B^{s-1})}
    +\frac{\mu}{2}\|v\|_{L^1_T(B^{s+1})}+
    \frac{\mu\|E\|_{L^\infty_T(B^{s})}}{2C}
     \nonumber\\
          &\leq&C
       \|E_0\|_{{B}^{s}}+C\|v_0\|_{B^{s-1}}+C
      \int^T_0 (\|\nabla v\|_{B^\frac{N}{2}}+1)
       \|E\|_{{B}^{s}}dt.\nonumber
  \end{eqnarray}
Using Gronwall's inequality, we can obtain (\ref{vis-E1.8}).

\subsection{Uniqueness}

Let $(v^1,E^1,p^1)$ and $(v^2,E^2,p^2)$ are solutions of
(\ref{vis-E1.4}) satisfying $E^i\in C([0,T]; B^\frac{N}{2})$,
$v^i\in C([0,T];B^{\frac{N}{2}-1})\cap L^1(0,T; B^\frac{N}{2})$.
From Propositions \ref{vis-P3.1} and \ref{vis-P4.1}, we get
    \begin{eqnarray}
     && \|E^{i}-E_0\|_{\widetilde{L}^\infty_T({B}^{\frac{N}{2}-1})}\nonumber\\
     &\leq& e^{C\widetilde{V}^i(T)}\left(
       \int^T_0 \|\nabla v^iE^{i}-v^i\cdot\nabla E_0+\nabla v^i\|_{{B}^{\frac{N}{2}-1}}dt
      \right)\nonumber\\
        &\leq& C \big(
      \|\nabla v^i\|_{L^1(B^{\frac{N}{2}-1})}\|E^{i}\|_{{L}^\infty_T(
      {B}^{\frac{N}{2}})}+\| v^i\|_{L^1(B^{\frac{N}{2}})}\|
      \nabla E_0\|_{{L}^\infty_T(
      {B}^{\frac{N}{2}-1})}
       \nonumber\\
      &&+\|v^i\|_{L^1(B^\frac{N}{2})}
      \big)\nonumber\\
      &\leq&  C  T^\frac{1}{2}.\label{vis-E5.18}
    \end{eqnarray}
        Let $\tilde{v}$ be the solution of the linear heat equations,
    \begin{equation}
      \left\{\begin{array}{l}
        \partial_t\tilde{v}_l-\mu\Delta\tilde{v}_l=E_{0jk}
        \partial_j E_{0lk}
        +\partial_j E_{0lj},\\
        \tilde{v}(0,x)=0.
      \end{array}
      \right.
    \end{equation}
    It is easy to obtain that
$\tilde{v}\in C([0,T];B^{\frac{N}{2}-1} )$ and
    \begin{equation}
      \|\tilde{v}\|_{L^r_T(B^{\frac{N}{2}-1+\frac{2}{r}} )}
      \leq C \|E_{0jk}
        \partial_j E_{0lk}
        +\partial_j E_{0lj}\|_{L^1_T(B^{\frac{N}{2}-1} )}\leq CT,
      \ r\in[1,\infty] .
    \end{equation}
When $N\geq 3$, from Propositions \ref{vis-P3.1} and \ref{vis-P4.2},
we have
 \begin{eqnarray}
    &&\|v^i-\bar{v}-\tilde{v} \|_{\widetilde{L}^\infty_T(B^{\frac{N}{2}-2})}
    +\mu\|v^i-\bar{v}-\tilde{v} \|_{L^1_T(B^{\frac{N}{2}})}
     \nonumber\\
    &\leq& C  \|-v^i\cdot \nabla (\bar{v}+\tilde{v})_l+
    E^i_{jk}\partial_j E^i_{lk}-E_{0jk}
        \partial_j E_{0lk}+\partial_j(E_{lj}^i-E_{0lj})\|_{L^1_T(B^{\frac{N}{2}-2})}\nonumber\\
        &\leq&C \left(\| v^i\|_{L^\infty(B^{\frac{N}{2}-1})} \|\nabla(
         \bar{v}+\tilde{v})\|_{L^1(B^{\frac{N}{2}-1} )}
          \right.\nonumber\\
            &&\left.+
          T\|(E^i,E_0)\|_{L^\infty(B^{\frac{N}{2} })}\| E^i-E_0 \|_{L^\infty_T(B^{\frac{N}{2}-1})}
          + T\| E^i-E_0 \|_{L^\infty_T(B^{\frac{N}{2}-1})}\right)\nonumber\\
          &\leq&C(T) .\label{vis-E5.21}
  \end{eqnarray}
Let   $\delta v=v^1-v^2$ and $\delta E=E^1-E^2$. Then, we have that
 $$(\delta v,\delta E)\in C([0,T];(B^{\frac{N}{2}-2})^N\times (B^{\frac{N}{2}-1})^{N\times N}
    ).$$
  From (\ref{vis-E1.4}),
we have
\begin{equation}
  \left\{\begin{array}{l}
    \nabla \cdot \delta v=\nabla \cdot \delta E^\top=0,\\
        \delta v_{it}+v^1\cdot \nabla \delta v_i+\delta v\cdot\nabla
        v^2_i
        +\partial_i \delta p\\
            \ \ \ =
        \mu\Delta \delta v_i+ E^1_{jk}\partial_j \delta E_{ik}
        + \delta E_{jk}\partial_j  E^2_{ik}+\partial_j\delta E_{ij},\\
           \delta  E_t+v^1\cdot \nabla \delta E+
           \delta v\cdot\nabla E^2=\nabla v^1\delta E+\nabla \delta v E^2
           +\nabla \delta v,\\
                (\delta v,\delta E)(0,x)=(0,0).
  \end{array} \right.\label{vis-E6.2}
\end{equation}

From Propositions \ref{vis-P3.1} and \ref{vis-P4.1}, we get
    \begin{eqnarray}
     && \|\delta E \|_{\widetilde{L}^\infty_T({B}^{\frac{N}{2}-1})}\nonumber\\
     &\leq& e^{C\widetilde{V}^1(T)}\int^T_0 \|-\delta v\cdot\nabla E^2+\nabla v^1\delta E
     +\nabla \delta v E^2+\nabla \delta v\|_{{B}^{\frac{N}{2}-1}}dt
      \nonumber\\
        &\leq&C\int^T_0(\|\delta v\|_{B^{\frac{N}{2}}}\|
        E^2\|_{B^{\frac{N}{2} }}+
        \|\nabla v^1\|_{B^\frac{N}{2}}\|\delta
        E\|_{{B}^{\frac{N}{2}-1}}+\|\delta v\|_{B^{\frac{N}{2}}})dt
\nonumber\\
      &\leq& C\|\delta v\|_{L^1_T(B^{\frac{N}{2}})}+C\int^T_0
        \|\nabla v^1\|_{B^\frac{N}{2}}\|\delta
        E\|_{{B}^{\frac{N}{2}-1}}dt.\label{vis-E6.3}
    \end{eqnarray}
From Propositions \ref{vis-P3.1} and \ref{vis-P4.2}, we have
 \begin{eqnarray}
    &&\|\delta v \|_{\widetilde{L}^\infty_T(B^{\frac{N}{2}-2})}+\mu\|\delta v \|_{L^1_T(B^{\frac{N}{2}})}
    \label{vis-E6.4-1}\\
    &\leq& Ce^{C\widetilde{V}^1 (T)} \|-\delta v\cdot \nabla v^2_i+ E^1_{jk}\partial_j \delta E_{ik}
    +\delta E_{jk}\partial_j E^2_{ik}+\partial_j\delta E_{ij}\|_{L^1_T(B^{\frac{N}{2}-2})}\nonumber\\
        &\leq&C \int^T_0\left[\|\delta v\|_{B^{\frac{N}{2}-2}}\|\nabla
        v^2\|_{B^\frac{N}{2}}+(\|  E^1\|_{B^{\frac{N}{2}}}
        +\|  E^2\|_{B^{\frac{N}{2}}}+1)\|\delta E\|_{B^{\frac{N}{2}-1}}
      \right]dt . \nonumber
  \end{eqnarray}
From (\ref{vis-E6.3})--(\ref{vis-E6.4-1}), we have
\begin{eqnarray}
    &&\|\delta E \|_{\widetilde{L}^\infty_T({B}^{\frac{N}{2}-1})}+
    \|\delta v \|_{\widetilde{L}^\infty_T(B^{\frac{N}{2}-2})}+\|\delta v \|_{L^1_T(B^{\frac{N}{2}})}
     \label{vis-E6.4}\\
    &\leq& C \int^T_0\left[\|\delta v\|_{B^{\frac{N}{2}-2}}\|\nabla
        v^2\|_{B^\frac{N}{2}}+\|\nabla v^1\|_{B^\frac{N}{2}}\|\delta
        E\|_{{B}^{\frac{N}{2}-1}}+\|\delta E\|_{B^{\frac{N}{2}-1}}
      \right]dt .\nonumber
  \end{eqnarray}
Using Gronwall's inequality, we obtain that $\delta E=\delta v=0$
and finish the proof of uniqueness part of Theorem \ref{vis-T1.1}
when $N\geq 3$.

When $N=2$, using the similar arguments as that in the proof of
(\ref{vis-E5.18})--(\ref{vis-E5.21}), we have that $(\delta v,\delta
E)\in C([0,T]; (\dot{B}^{-1}_{2,\infty})^N\times
(\dot{B}^{0}_{2,\infty})^{N\times N}
    )$. From Propositions \ref{vis-P3.1} and \ref{vis-P4.1}, we get
    \begin{eqnarray*}
     && \|\delta E \|_{\widetilde{L}^\infty_t(\dot{B}^{0 }_{2,\infty})}\nonumber\\
     &\leq& e^{C\widetilde{V}^1(t)}\int^t_0 \|-\delta v\cdot\nabla E^2+\nabla v^1\delta E
     +\nabla \delta v E^2+\nabla \delta v\|_{\dot{B}^{0}_{2,\infty}}dt
      \nonumber\\
        &\leq&C\int^t_0(\|\delta v\|_{B^{1}}\|\nabla
        E^2\|_{\dot{B}^{0 }_{2,\infty}}+
        \|\nabla v^1\|_{B^1}\|\delta
        E\|_{\dot{B}^{0}_{2,\infty}}+\|\delta v\|_{\dot{B}^1_{2,\infty}})dt
\nonumber\\
      &\leq& C\|\delta v\|_{L^1_t(B^1)}+C\int^t_0
        \|\nabla v^1\|_{B^1}\|\delta
        E\|_{\dot{B}^{0}_{2,\infty}}ds,\ t\in[0,T].
    \end{eqnarray*}
Now, inserting the following logarithmic inequality  (see (4.6) in
 \cite{Danchin2003}) in the above estimate,
    \begin{equation}
      \|f\|_{L^1_T(\dot{B}^1_{2,1})}\lesssim \|f\|_{\widetilde{L}^1_T(\dot{B}^1_{2,\infty})}
      \ln(e+\frac{\|f\|_{\widetilde{L}^1_T(\dot{B}^0_{2,\infty})}+\|f\|_{\widetilde{L}^1_T(\dot{B}^2_{2,\infty})}}{
      \|f\|_{\widetilde{L}^1_T(\dot{B}^1_{2,\infty})}}),
    \end{equation}
    we have for $t\in[0,T]$,
        \begin{eqnarray*}
       \|\delta E \|_{\widetilde{L}^\infty_t(\dot{B}^{0 }_{2,\infty})}&\leq& C\int^t_0
        \|\nabla v^1\|_{B^1}\|\delta
        E\|_{\dot{B}^{0}_{2,\infty}}ds\nonumber\\
      & & +C\|\delta v\|_{L^1_t(\dot{B}^1_{2,\infty})}\ln(e+\frac{\|\delta v\|_{\widetilde{L}^1_t(\dot{B}^0_{2,\infty})}+\|\delta v\|_{\widetilde{L}^1_t(\dot{B}^2_{2,\infty})}}{
      \|\delta v\|_{\widetilde{L}^1_t(\dot{B}^1_{2,\infty})}})    .
    \end{eqnarray*}
    Since $$
    \|\delta v\|_{\widetilde{L}^1_t(\dot{B}^0_{2,\infty})}+\|\delta v\|_{\widetilde{L}^1_t(\dot{B}^2_{2,\infty})}
    \leq \sum_{i=1}^2(
        \| v^i\|_{\widetilde{L}^1_t(\dot{B}^0_{2,1})}+\|  v^i\|_{\widetilde{L}^1_t(\dot{B}^2_{2,1})})\leq C_T,
        $$
        we have for $t\in[0,T]$,
        \begin{eqnarray*}
      && \|\delta E \|_{\widetilde{L}^\infty_t(\dot{B}^{0 }_{2,\infty})}\nonumber\\
        &\leq&  C\int^t_0
        \|\nabla v^1\|_{B^1}\|\delta
        E\|_{\dot{B}^{0}_{2,\infty}}ds +C\|\delta v\|_{L^1_t(\dot{B}^1_{2,\infty})}\ln(e+ \frac{C_T }{
      \|\delta v\|_{\widetilde{L}^1_t(\dot{B}^1_{2,\infty})}}),
    \end{eqnarray*}
    and using Gronwall's inequality,
            \begin{equation}
       \|\delta E \|_{\widetilde{L}^\infty_t(\dot{B}^{0 }_{2,\infty})} \leq  C\|\delta v\|_{L^1_t(\dot{B}^1_{2,\infty})}\ln(e+ \frac{C_T }{
      \|\delta v\|_{\widetilde{L}^1_t(\dot{B}^1_{2,\infty})}})
       .\label{vis-E6.3-00}
    \end{equation}
From Propositions \ref{vis-P3.1} and \ref{vis-P4.2}, we have
 \begin{eqnarray}
    &&\|\delta v \|_{\widetilde{L}^\infty_t(\dot{B}^{-1}_{2,\infty})}
    +\mu\|\delta v \|_{L^1_t(\dot{B}^{1}_{2,\infty})}
     \nonumber\\
    &\leq& Ce^{C\widetilde{V}^1 (t)} \|-\delta v\cdot \nabla v^2_i+ E^1_{jk}\partial_j \delta E_{ik}
    +\delta E_{jk}\partial_j E^2_{ik}+\partial_j\delta E_{ij}\|_{L^1_t(\dot{B}^{-1}_{2,\infty})}\nonumber\\
        &\leq&C \int^t_0\left[\|\delta v\|_{\dot{B}^{-1}_{2,\infty}}\|\nabla
        v^2\|_{B^1}+(\|  E^1\|_{B^1}
        +\|  E^2\|_{B^1}+1)\|\delta E\|_{\dot{B}^{0}_{2,\infty}}
      \right]ds \nonumber\\
        &\leq&C \int^t_0\left[\|\delta v\|_{\dot{B}^{-1}_{2,\infty}}\|\nabla
        v^2\|_{B^1}+ \|\delta E\|_{L^\infty_t(\dot{B}^{0}_{2,\infty})}\right]ds.
    \label{vis-E6.4-00}
  \end{eqnarray}
Denote
    $$
    W(t):=\|\delta v \|_{\widetilde{L}^\infty_T(\dot{B}^{-1}_{2,\infty})}
    +\mu\|\delta v \|_{L^1_T(\dot{B}^{1}_{2,\infty})}.
    $$
From (\ref{vis-E6.3-00})--(\ref{vis-E6.4-00}), we have
    $$
    W(t)\leq C\int^t_0 (1+\|\nabla
        v^2\|_{B^1})W(s)\ln(e+\frac{C_T}{W(s)})ds,\ T\in[0,t].
    $$
    As $$
    \int^1_0\frac{dr}{r\ln(e+\frac{C_T}{r})}=\infty,$$
a slight generalization of Gronwall lemma (see e.g. lemma 3.1 in
\cite{Chemin1992}) implies that $W\equiv0$ on $[0,T]$. Thus we
obtain that  $\delta E=\delta v=0$, and finish the proof of
uniqueness part of Theorem \ref{vis-T1.1}
 when $N=2$. {\hfill
$\square$\medskip}

\section{A global existence and uniqueness result}\label{vis-s6}
This section is devoted to the proof of the following theorem.
\begin{thm}\label{vis-T6.1}
  Consider the viscoelastic model (\ref{vis-E1.4}). Suppose that the
   initial data satisfies the incompressible constraints
   (\ref{vis-E1.3-0})--(\ref{vis-E1.3}),
    $E_0\in B^\frac{N}{2}\cap B^{\frac{N}{2}-1}$, $v_0\in   B^{\frac{N}{2}-1}$ and
        \begin{equation}
          \|E_0\|_{B^\frac{N}{2}\cap B^{\frac{N}{2}-1}}+\|v_0\|_{B^{\frac{N}{2}-1}}\leq \lambda,
        \end{equation}
        where $\lambda$ is a small positive constant.
  Then there
  exists a  global   solution for system (\ref{vis-E1.4}) that satisfies
    \begin{equation}
      \|(v,E)\|_{X^\frac{N}{2}}\leq M(\|E_0\|_{B^\frac{N}{2}\cap
      B^{\frac{N}{2}-1}}+\|v_0\|_{B^{\frac{N}{2}-1}}),
    \end{equation}
  where
   \begin{eqnarray*}
    X^s=\big\{
    (v,c)\in &&\left(L^1(\mathbb{R}^+; {B}^{s+1} )\cap C(\mathbb{R}^+; {B}^{s-1} )
    \right)^N\\
        &&\times\left(L^1(\mathbb{R}^+;\tilde{B}^{s,1}_\mu)\cap C(\mathbb{R}^+;\tilde{B}^{s,\infty}_\mu)
    \right)^N
    \big\}.
    \end{eqnarray*}
\end{thm}

Denote
    $$
    \alpha=\|E_0\|_{\tilde{B}^{\frac{N}{2},\infty}_\mu}+\|v_0\|_{B^{\frac{N}{2}-1}}.
    $$
From  Theorem \ref{vis-T1.1}, we have that there exists a unique
local solution $(v,E)$ of (\ref{vis-E1.4}) with the initial data
$(v_0,E_0)$. Assume that the maximum existence time is $T^*$, such
that the solution $(v,E)$ exists on  $ [0,T^*)$ and
    $$
    (v,E)\in  \left(L^1([0,T^*); B^{\frac{N}{2}+1})
    \cap C([0,T^*);   B^{\frac{N}{2}-1})
    \right)^N\times\left(
C([0,T^*); B^\frac{N}{2})
    \right)^{N\times N}.
    $$
Using Proposition \ref{vis-P3.1}, we can easily obtain that $E\in
\left( C([0,T^*); B^{\frac{N}{2}-1})
    \right)^{N\times N}$ and omit the details.
From Lemmas \ref{vis-L2.2}--\ref{vis-L2.3}, we have
    \begin{equation}
    \partial_m E_{ij}-\partial_jE_{im}
       =E_{lj}\partial_lE_{im}-E_{lm}\partial_lE_{ij}
       =\partial_l(E_{lj}E_{im}-E_{lm}E_{ij}).\label{vis-E7.1}
  \end{equation}

 We are going to prove the existence of a positive
$M$ such that, if $\alpha$ is small enough, the following bound
holds,
        \begin{equation}
          \|(v,E)\|_{X^\frac{N}{2}_{T^*}}\leq
          M\alpha.\label{vis-E7.2-0}
        \end{equation}

\textbf{Claim 1.} If
        \begin{equation}
          \|(v,E)\|_{X^\frac{N}{2}_{T}}\leq 2M\alpha,\
          T\in(0,T^*),\label{vis-E7.2}
        \end{equation}
then, we have
    \begin{equation}
          \|(v,E)\|_{X^\frac{N}{2}_{T}}\leq M\alpha.
        \end{equation}
when $\alpha$ is small enough.

 Let $c=\Lambda^{-1}\nabla\cdot E$, we
have
   \begin{equation}
  \left\{\begin{array}{l}
    \nabla \cdot v=\nabla \cdot c=0,  \\
                               v_t+v\cdot \nabla v+\nabla p-\mu\Delta v-\Lambda c=G,\\
                                c_t+v\cdot \nabla c+\Lambda v =L,\\
                     \Delta E_{ij}= \Lambda\partial_j c_i+\partial_k(\partial_k E_{ij}-\partial_jE_{ik }),\\
                (v,c)(0,x)=(v_0,c_0)(x),
  \end{array}
  \right.\label{vis-E5.1}
\end{equation}
with
    $$
    L_i=\Lambda^{-1}\partial_j(\nabla v E)_{ij}-[\Lambda^{-1}\partial_j, v\cdot]\nabla E_{ij},
    \ G_{i}=E_{jk}\partial_jE_{ik}.
    $$
From Proposition \ref{vis-P4.3}, we have
    \begin{eqnarray}
       \|(v,c) \|_{X^\frac{N}{2}_T}
       &\leq& Ce^{C\|v\|_{L^1_T(B^{\frac{N}{2}+1})}}\left(
      \|c_0\|_{\tilde{B}^{\frac{N}{2},\infty}_\mu}
      +\|v_0\|_{B^{\frac{N}{2}-1}}\right.\nonumber\\
        &&\left.+
      \|L(s)\|_{L^1_T(\tilde{B}^{\frac{N}{2},\infty}_\mu)}+
      \|G(s)\|_{L^1_T(
      \tilde{B}^{\frac{N}{2}-1}) }
      \right).\label{vis-E7.5}
    \end{eqnarray}
From  Proposition \ref{vis-P3.2},  Lemma \ref{vis-L3.1}  and
(\ref{vis-E7.2}), we obtain
    \begin{eqnarray}
\|L\|_{L^1_T(\tilde{B}^{\frac{N}{2},\infty}_\mu)}
&\leq&\|\Lambda^{-1}\partial_j(\nabla v
E)_{ij}\|_{L^1_T(\tilde{B}^{\frac{N}{2},\infty}_\mu)}
   +\|[\Lambda^{-1}\partial_j, v\cdot]\nabla E_{ij}\|_{L^1_T(\tilde{B}^{\frac{N}{2},\infty}_\mu)}
   \nonumber\\
&    \leq & C \| \nabla v E
\|_{L^1_T(\tilde{B}^{\frac{N}{2},\infty}_\mu)}
            +C \|E\|_{L^\infty_T(\tilde{B}^{\frac{N}{2},\infty}_\mu)}\|\nabla v\|_{L^1_T
                (B^\frac{N}{2})}\nonumber\\
                &\leq&C \|E\|_{L^\infty_T(\tilde{B}^{\frac{N}{2},\infty}_\mu)}\|\nabla v\|_{L^1_T
                (B^\frac{N}{2})}\leq CM^2\alpha^2.
    \end{eqnarray}
 From Proposition \ref{vis-P3.1} and (\ref{vis-E7.2}), we get
        \begin{eqnarray}
        && \|G(s)\|_{L^1_T(
      \tilde{B}^{\frac{N}{2}-1}) }
      \leq  C\|E\|_{L^2_T(B^\frac{N}{2})}\|\nabla E\|_{L^2_T(B^{\frac{N}{2}-1})}
      \leq C\|E\|_{L^2_T(B^\frac{N}{2})}^2
      \nonumber\\
        &=&C\int^T_0\left(\sum_{q\in\mathbb{Z}}\left(2^{q\frac{N}{2}}\|\Delta_q E\|_{L^2}\max\{\mu,2^{-q}\}
        \right)^\frac{1}{2}\right.\nonumber\\
        &&\times\left.\left(2^{q\frac{N}{2}}\|\Delta_q E\|_{L^2}\min\{\mu^{-1},2^{ q}\}
        \right)^\frac{1}{2}
        \right)^2dt\nonumber\\
            &\leq& C\|E\|_{L^\infty_T(\tilde{B}^{\frac{N}{2},\infty}_\mu)}
            \|E\|_{L^1_T(\tilde{B}^{\frac{N}{2},1}_\mu)}\leq CM^2\alpha^2.\label{vis-E7.7}
        \end{eqnarray}
 From (\ref{vis-E7.5})--(\ref{vis-E7.7}), we have
    \begin{eqnarray}
       \|(v,c) \|_{X^\frac{N}{2}_T}
       &\leq& Ce^{CM\alpha}\left(
      \|E_0\|_{\tilde{B}^{\frac{N}{2},\infty}_\mu}
      +\|v_0\|_{B^{\frac{N}{2}-1}}+
     CM^2\alpha^2
      \right).\label{vis-E7.9}
    \end{eqnarray}

From (\ref{vis-E7.1}) and (\ref{vis-E5.1})$_4$, we have
    \begin{eqnarray}
      \|E\|_{L^\infty_T(\widetilde{B}^{\frac{N}{2},\infty}_\mu)}&\leq&
      C\|c\|_{L^\infty_T(\widetilde{B}^{\frac{N}{2},\infty}_\mu)}+
      C\|E^2\|_{L^\infty_T(\widetilde{B}^{\frac{N}{2},\infty}_\mu)}\nonumber\\
    &\leq& C \|c\|_{L^\infty_T(\widetilde{B}^{\frac{N}{2},\infty}_\mu)}+
      C\|E\|_{L^\infty_T(B^{\frac{N}{2}})}
      \|E\|_{L^\infty_T(\widetilde{B}^{\frac{N}{2},\infty}_\mu)}\nonumber\\
    &\leq& C \|c\|_{L^\infty_T(\widetilde{B}^{\frac{N}{2},\infty}_\mu)}+
     CM^2\alpha^2
    \end{eqnarray}
and
    \begin{eqnarray}
      \|E\|_{L^1_T(\widetilde{B}^{\frac{N}{2},1}_\mu)}&\leq&
      C\|c\|_{L^1_T(\widetilde{B}^{\frac{N}{2},1}_\mu)}+
      C\|E^2\|_{L^1_T(\widetilde{B}^{\frac{N}{2},1}_\mu)}\nonumber\\
      &\leq&
      C\|c\|_{L^1_T(\widetilde{B}^{\frac{N}{2},1}_\mu)}+
      C\|E^2\|_{L^1_T(B^{\frac{N}{2}})}\nonumber\\
    &\leq& C \|c\|_{L^1_T(\widetilde{B}^{\frac{N}{2},1}_\mu)}+
      C\|E\|_{L^2_T(B^{\frac{N}{2}})}^2\nonumber\\
    &\leq& C \|c\|_{L^1_T(\widetilde{B}^{\frac{N}{2},1}_\mu)}+
      CM^2\alpha^2.\label{vis-E7.11}
    \end{eqnarray}
From (\ref{vis-E7.9})--(\ref{vis-E7.11}), we have
    \begin{eqnarray}
       \|(v,E) \|_{X^\frac{N}{2}_T}
       &\leq& Ce^{CM\alpha}\left(
      \|E_0\|_{\tilde{B}^{\frac{N}{2},\infty}_\mu}
      +\|v_0\|_{B^{\frac{N}{2}-1}}+
     CM^2\alpha^2
      \right)+CM^2\alpha^2\nonumber\\
        &\leq& 2C\alpha+ \alpha\leq M\alpha.
    \end{eqnarray}
when $M=2C+1$ and $\alpha$ satisfies
    \begin{equation}
e^{CM\alpha}\leq 2,\ 2C^2M^2\alpha+CM^2\alpha\leq 1.
    \end{equation}
Then, we finish the proof of Claim 1. From the classical
continuation method and Claim 1, we can easily obtain that
(\ref{vis-E7.2-0}). Combining Theorem \ref{vis-T1.1}, one can obtain
that $T^*=\infty$ and finish the proof of Theorem \ref{vis-T1.2}.  {\hfill
$\square$\medskip}

\end{document}